\documentclass[12pt]{article}

\usepackage{amsmath,amsthm,amsfonts,amssymb,amscd,epsf,epsfig,psfrag,enumerate}
\usepackage[mathscr]{eucal}

\title{Essential disks and semi-essential surfaces in
3-manifolds}

\author{Charalampos Charitos and Ulrich Oertel}

\date{May 5, 2010}

\newtheorem{thm}{Theorem}[section] \newtheorem{lemma}[thm]{Lemma}
 
\newtheorem{corollary}[thm]{Corollary}

\newtheorem{proposition}[thm]{Proposition}
 \newtheorem*{claim*}{Claim}

 \theoremstyle{definition}
\newtheorem{defn}[thm]{Definition}
 
 \newtheorem{ex}[thm]{Example}
 
\theoremstyle{remark}

\begin{document}

\maketitle

\centerline{\bf $\,$}
\vskip 0.5in

\def\HDS{half-disk sum}

\def\Length{\text{Length}}

\def\Area{\text{Area}}
\def\Im{\text{Im}}
\def\im{\text{Im}}
\def\rel{\text{ rel }}
\def\irred{irreducible}
\def\half{spinal pair }
\def\spinal{\half}
\def\spinals{\halfs}
\def\halfs{spinal pairs }
\def\reals{\mathbb R}
\def\rationals{\mathbb Q}
\def\complex{\mathbb C}
\def\naturals{\mathbb N}
\def\integers{\mathbb Z}
\def\id{\text{id}}

\def\proj{P}
\def\hyp {\hbox {\rm {H \kern -2.8ex I}\kern 1.15ex}}

\def\Diff{\text{Diff}}

\def\weight#1#2#3{{#1}\raise2.5pt\hbox{$\centerdot$}\left({#2},{#3}\right)}
\def\intr{{\rm int}}
\def\inter{\ \raise4pt\hbox{$^\circ$}\kern -1.6ex}
\def\Cal{\cal}
\def\from{:}
\def\inverse{^{-1}}
\def\Max{{\rm Max}}
\def\Min{{\rm Min}}
\def\fr{{\rm fr}}
\def\embed{\hookrightarrow}
\def\Genus{{\rm Genus}}
\def\Z{Z}
\def\X{X}

\def\roster{\begin{enumerate}}
\def\endroster{\end{enumerate}}
\def\intersect{\cap}
\def\definition{\begin{defn}}
\def\enddefinition{\end{defn}}
\def\subhead{\subsection\{}
\def\theorem{thm}
\def\endsubhead{\}}
\def\head{\section\{}
\def\endhead{\}}
\def\example{\begin{ex}}
\def\endexample{\end{ex}}
\def\ves{\vs}
\def\mZ{{\mathbb Z}}
\def\M{M(\Phi)}
\def\bdry{\partial}
\def\hop{\vskip 0.15in}
\def\hip{\vskip0.05in}
\def\mathring{\inter}
\def\trip{\vskip 0.09in}
\def\PML{\mathscr{PML}}
\def\H{\mathscr{H}}
\def\C{\mathscr{C}}
\def\S{\mathscr{S}}
\def\T{\mathscr{T}}
\def\E{\mathscr{E}}
\def\K{\mathscr{K}}
\def\BL{\mathscr{BL}}
\def\L{\mathscr{L}}
\def\suchthat{|}
\newcommand\invlimit{\varprojlim}
\newcommand\congruent{\equiv}
\newcommand\modulo[1]{\pmod{#1}}
\def\ML{\mathscr{ML}}
\def\Stack{\mathscr{T}}
\def\M{\mathscr{M}}
\def\A{\mathscr{A}}
\def\union{\cup}
\def\atlas{\mathscr{A}}
\def\interior{\text{Int}}
\def\frontier{\text{Fr}}
\def\composed{\circ}

\begin{abstract} If $M$ is a manifold with compressible boundary, we analyze
essential disks in $M$, as well as incompressible, but not necessarily
$\bdry$-incompressible, surfaces in $M$.   We are most interested in the case where
$M$ is a handlebody or compression body.  The analysis depends on a new normal
surface theory.  We hope the normal surface theory will be used in other papers to
describe objects representing limits of essential disks in a handlebody or a
3-manifold with compressible boundary.  For certain automorphisms of handlebodies,
these disk limits should serve as invariant objects akin to laminations and analogous to
the invariant laminations for pseudo-Anosov automorphisms of surfaces.
\end{abstract}

\section{Introduction}\label{Intro}  Let $H$ be a 3-dimensional handlebody of genus
$g$.  In
\cite{HM:HandlebodyMapClass}, Masur describes the limit set $\L$ of the action of
the mapping class group on the projective lamination space $\PML(\bdry H)$ of $\bdry
H$.  It is the closure in $\PML(\bdry H)$ of points represented by simple closed
curves bounding disks in
$H$.  Clearly it would be interesting to know whether every point in $\L$ ``bounds" 
some naturally defined object, similar to a measured lamination, and generalizing systems of
essential disks.  This is more than an idle theoretical question:  There are certain
automorphisms (self-homeomorphisms) of $H$ called {\it generic} automorphisms, similar to
pseudo-Anosov automorphisms, which restrict to pseudo-Anosov automorphisms on $\bdry H$, see
\cite{UO:Autos}.  In \cite{UO:Autos} and \cite{LC:Tightness}, invariant
measured laminations for these automorphisms are described and analyzed, but these
laminations are still poorly understood.  If
$f:H\to H$ is a generic automorphism, then the induced $\bdry f:\bdry H\to \bdry H$ is
pseudo-Anosov, with stable and unstable invariant measured laminations, $L_+$ and
$L_-$.  These are fixed points of the action of $\bdry f:\bdry H\to \bdry H$ on
$\PML(\bdry H)$ and must lie in $\L$.  If one can understand the ``laminations"
bounded by $L_+$ and
$L_-$ in $H$, these should be closely related to the known invariant laminations for
$f$.    This gives some motivation for the more basic task we set ourselves in this
paper.  Here we will only attempt to develop a normal surface theory, and a related
theory of carriers, analogous to the theory of branched surfaces, useful for
describing the essential disks and incompressible surfaces in an irreducible
3-manifold $M$ with compressible boundary.  Of greatest interest is the case where
$M$ is a handlebody, or a compression body.  The meaning of {\it essential disk} is
clear; this a disk whose boundary does not bound a disk in
$\bdry M$.  We say a surface $S$ properly embedded in $M$ is {\it semi-essential} if
every component of $S$ is (i) an essential disk, or (ii) an incompressible, possibly
$\bdry$-compressible surface.  Semi-essential surfaces which consist of
$\bdry$-parallel annuli are not interesting, but we allow such annuli as components
of semi-essential surfaces.   We choose a triangulation $\Delta$ for $M$.  Then we
modify the triangulation as follows.  If $F=\bdry M$, the triangulation $\Delta$
induces a triangulation of $F$.  Then we decompose
$F\times I$ into prisms $\sigma \times I$, where we have one prism for every
2-simplex $\sigma$ of the induced triangulation of $F$.  Finally, we attach $F\times
I$ to $\bdry M$ by identifying $F\times 1$ with $\bdry M$.  The new manifold $M\cup
F\times I$ is homeomorphic to $M$, so we identify it with $M$.  It is ``cellulated"
by prisms of the form
$\sigma\times I$ and  3-simplices.  Henceforth, we let $\C$ denote this modified
triangulation which we call a {\it cellulation}.  We will call the prisms and
simplices (of any dimension) in this cellulation {\it cells} of the cellulation. 
The prisms of the cellulation are actually pairs $(\Sigma,\sigma)$, where $\sigma$
is a 2-simplex in $\bdry M$. For prisms of the form $\sigma\times I$, where $\sigma$
is a 2-simplex, we make the convention that
$\sigma\times 0$ lies in $\bdry M$.  

Properly embedded disks in any 3-cell $\Sigma$ of the cellulation belong to {\it
disk types}, where properly embedded disks $E_1$  and $E_2$ in $\Sigma$ have the
same type if there is a homeomorphism $\Sigma \to
\Sigma$ taking
$E_1$ to
$E_2$, mapping each vertex to itself, mapping each edge to itself, and mapping each
2-cell to itself homeomorphically.  We will define an infinite number of {\it tunnel
normal disk types} for a 3-prism
$(\Sigma,\sigma)$, and we will define a finite number of {\it tunnel normal disk
types} for a 3-simplex.  Any disk embedded in a 3-cell is a {\it tunnel normal disk}
if it belongs to a tunnel normal disk type.  A properly embedded surface $S\embed M$
is in {\it tunnel normal form} with respect to a cellulation if it intersects each
3-cell in a collection of disks, each belonging to a tunnel normal disk type.

\begin{thm}\label{NormalThm}  Let $M$ be an irreducible, orientable, compact
3-manifold with a cellulation
$\C$.  Every semi-essential surface $S$ properly embedded in $M$ can isotoped to a
tunnel normal form with respect to the cellulation, such that it intersects each
3-cell in a disjointly embedded union of tunnel normal disks. 
\end{thm}

The details of the normal surface theory will be given in Section \ref{Normal}.  We
will define a complexity which represents a kind of combinatorial area, then normal
surfaces of minimal complexity represent minimal surfaces in this sense.  We will
also describe {\it carriers}, which generalize branched surfaces.  In particular, a
tunnel normal semi-essential surface of a minimal complexity can be shown to be
carried by a certain kind of carrier, called an {\it essential carrier}.  One can
deduce from Theorem
\ref{NormalThm} the following statement involving carriers.

\begin{thm}\label{FiniteThm}    Let $M$ be an irreducible, orientable, compact
3-manifold. There exist finitely many essential carriers $C_i,\ i=1,2,\ldots,m$ in
$M$ such that: 

\noindent (i) Every (isotopy class of a) semi-essential surface in $M$ is carried by
a carrier $C_i$. 

\noindent (ii) Every (isotopy class of an) essential curve system $L$ in $\bdry M$ is
carried by the train track $\bdry C_i=C_i\cap \bdry M$ for some carrier $C_i$, and
every curve system $L$ carried by a
$\bdry C_i$ is essential.

\noindent (iii) Every system $S$ of disks in $M$ carried by a carrier $C_i$ in the
collection is essential, i.e. every disk in the system $S$ is essential.
\end{thm}

It is not possible to prove that every surface carried by an essential carrier is
semi-essential.  In any case, non-disk, non-essential, semi-essential surfaces seem
much less important than essential disks.  Nevertheless, we want some understanding
of these surfaces.  It turns out that one must study them by subdividing the set of
all semi-essential surfaces into smaller subclasses.  We do this as follows.  We fix
a cell in the curve complex of
$\bdry M$.  This amounts to choosing a {\it primitive curve system} $X$, i.e. a
curve system
$X$ with the property that no two curves of the system are isotopic.  We now
restrict our attention to the set of surfaces $\S_X$ of all semi-essential surfaces
$S$ with $\bdry S\subset N(X)$, where $N(X)$ is a regular neighborhood of $X$ in
$\bdry M$.  There is a characteristic decomposition
$M=T_X\cup Q_X$ corresponding to this choice of $X$, where $Q_X$ is a certain
compression body, and $T_X$ is the remainder of the manifold.  $T_X$ is called the
$X$-core of
$M$ and $Q_X$ is called the $X$-characteristic compression body.  See Section
\ref{Non-disk} for details.  The
$X$-characteristic compression body is actually a pair $(Q_X,V)$, where $V$ is a
properly embedded surface in
$M$ along which one cuts $M$ to obtain $Q_X$ and $T_X$.  The exterior boundary of
$Q_X$ is denoted $W$; it is the closure in $\bdry Q_X$ of $\bdry Q_X-V$.  We also
sometimes view $Q_X$ as a pair $(Q_X,W)$, though $W$ is not in general
incompressible in $Q_X$.  The
$X$-core should be viewed as a pair
$(T_X,A)$; namely,
$V$ also appears as a subsurface of $\bdry T_X$, and $A$ is the complementary
surface in $\bdry T_X$, whose components are all annuli.

Recall a {\it half-disk} is a pair $(H,\alpha)$, where $H$ is a disk and $\alpha$ is
a closed arc in $\bdry H$.  Let $\beta$ be the complementary arc in $\bdry H$.   As
usual, we say a surface $(S,\bdry S)\embed (T_X,A)$ is {\it $\bdry$-incompressible}
if for every half-disk $(H,\alpha)$ embedded in
$(M,A)$ with $H\cap S=\beta$, there is a half-disk cut from $S$ by $\beta$.  The
definition of incompressibility is the usual one.

\begin{thm} \label{SeparationThm} Let $X$ be a primitive curve system in $\bdry M$. 
A surface $S$ of $\S_X$ is a union
$S=F\cup K$, where $K$ is a union of disks and $F$ has no disk components.  Then

\hip

\noindent (i) $(K,\bdry K)$ can be isotoped to a collection of essential disks in the
$X$-compression body $(Q_X,W)$, and  

\noindent (ii) $(F,\bdry F)$ can be isotoped to a
semi-essential surface in $(T_X,A)$.
\end{thm}

We can then analyze the semi-essential surfaces in $(T_X,A)$ using ordinary
incompressible branched surfaces: 
\begin{thm} \label{Non-DiscThm} Let $X$ be a primitive curve system in $\bdry M$.

\hip
\noindent (i) Every semi-essential surface $(S,\bdry S)\in \S_X$ in
$(T_X,A)$ which does not contain components which are
$\bdry$-parallel annuli is $\bdry$-incompressible (and incompressible) in $(T_X,A)$.

\noindent (ii) There exist finitely many incompressible branched surfaces $(B_{X,j},\bdry
B_{X,j})$ in $(T_X,A)$ such that every incompressible, $\bdry$-incompressible
surface $S$ in $(T_{X},A)$ is fully carried by one of the branched surfaces
$B_{X,j}$.  (This means every semi-essential surface in $M$ of $\S_X$ without
components which are essential disks or $\bdry$-parallel annuli is fully carried by
one of the $B_{X,j}$.)  (iii) If a surface is carried by one of the branched
surfaces $B_{X,j}$, it is incompressible and
$\bdry$-incompressible.
\end{thm}

There is another point of view which might be preferable in some situations.  If we
focus on the complex of curves in $\bdry M$ which do not bound disks in $M$, the
mapping class group of $M$ acts on this complex, so it is a natural subcomplex of
the curve complex.  A cell of this complex corresponds to a primitive system $Y$ of
curves in $\bdry M$, none of which bounds a disk in
$M$.  Let $\S_Y$ denote the set of surfaces
$S$ with the property that $\bdry S \subset N(Y)$.   There is a decomposition
$M=T_Y\cup Q_Y$ as before.  Similarly, if $Z$ is a curve system with the property
that each curve in $Z$ bounds a disk, and $S_Z$ is the set of surfaces $S$ with
$\bdry S\subset N(Z)$, then $Q_Z$ is the entire characteristic
compression body for
$M$.

\begin{corollary} \label{SeparationTwoCor} (a) Let $Y$ be a primitive curve system
in $\bdry M$ none of whose curves bounds a disk in $M$.  Then a surface
$S$ of
$\S_Y$   can be isotoped to a semi-essential surface in $(T_Y,A)$.

\noindent (b) Let $Z$ be a primitive curve system in $\bdry M$ all of whose curves
bound disks in $M$.  Then a surface
$S$ of
$\S_Z$  is a union of essential disks and can be isotoped to a system of disks in
the characteristic compression body $(Q,V)$ for $M$.
\end{corollary}

In statement (a) of the corollary, the semi-essential surfaces in $(T_Y,A)$ can again
be analyzed using Theorem \ref{Non-DiscThm}.  In case $M$ is a handlebody, note that
a cell of the curve complex of $\bdry M$ corresponds to a primitive curve system $Y$
with the property that no curve of $Y$ bounds a disk in $M$ if and only if the cell
is disjoint from Masur's limit set.  

For a general $M$, and a ``random" primitive curve system $X\subset \bdry M$, we
have $\bdry M-N(X)$ incompressible in $M$.  In this case we can take
$(T_X,A)=(M,A)$, and the semi-essential surfaces in $(M,A)$ can be analyzed using
incompressible branched surfaces.  
\hop

\noindent {\bf Plans:}  This paper should point the way towards further research, some of
which we are already pursuing actively. 

\noindent (1)  We have already mentioned using our normal surface theory to describe
disk-limit ``laminations."

\noindent (2)  Show that generic automorphisms of handlebodies have invariant disc-limit
``laminations."

\noindent (3)  One possible way to describe disc-limit ``laminations" is to take limits in a
solutions space analogous to the solution space for conventional normal surfaces.  This a
``space of normal surfaces."  It seems that for our normal surface theory, if it is possible
to describe a solution space, it will involve parameters in certain Hilbert spaces.

\noindent (4)  We have begun to investigate a normal surface theory for essential
spheres and surfaces in reducible 3-manifolds.  Conventional normal surface theory detects
some essential spheres, but cannot detect all of them.

\section{Tunnel normal form}\label{Normal}

The purpose of this section is to prove Theorem \ref{NormalThm}.

Let $M$ be an irreducible but not necessarily $\bdry$-irreducible, orientable
3-manifold. Suppose $S\embed M$ is a  semi-essential surface, i.e. a surface whose
components are either essential disks or incompressible surfaces.  We need only
incompressibility, not $\bdry$-incompressibility.    We choose a cellulation $\C$ of
$M$.  Recall from the introduction that a cellulation is a triangulation modified
by  replacing each 2-simplex $\sigma$ in $\bdry M$ by a prism $\sigma\times [0,1]$,
with $\sigma\times 0\subset \bdry M$.  If $\epsilon$ is an edge in $\bdry \sigma$,
we say $\epsilon\times I$ is a {\it 2-prism}.  If $\rho$ is a vertex of $\sigma$, we
say $\rho\times I$ is a {\it 1-prism or prism edge.}  We refer to the cells of
different dimensions as {\it boundary cells} or {\it interior cells} according to
whether they are entirely contained in $\bdry M$. If a cell is entirely contained in $\bdry
M$ it is called a boundary cell, otherwise it is called an interior cell.  Suppose
$S$ is a semi-essential surface.  First isotope
$S$ to general position with respect to the cellulation.  We minimize the number of
intersections
$|S\cap
\C^1|$ with the 1-skeleton $\C^1$ of the cellulation.  We also use the
irreducibility assumption to eliminate trivial closed curves of intersection in
interior faces.  By the minimality, we never have a trivial innermost arc of
intersection with a 2-simplex if both ends of the arc lie on an interior edge. 
However, we may have trivial arcs of intersection on a rectangle or 2-prism which
intersects
$\bdry M$ in an edge; for such an arc both ends are in a boundary edge.  We can also
easily conclude, as in standard normal surface theory, that for every 3-simplex
or 3-prism $\Sigma$ every component of
$S\cap
\Sigma$ is a disk.  We are left with a prenormal surface as defined below.

A {\it prenormal surface} with respect to the cellulation $\C$ is a properly
embedded surface $S$ which satisfies:  (1)
$S$ is in general position with respect to the cellulation.  (2) Components of
intersection of $S$ with each 3-simplex are disks. (3) Intersections with
2-dimensional faces contain no closed curves. 

We have shown above that every essential surface $S$ can be isotoped to a prenormal
surface.

\begin{defn}For prenormal surfaces we define a lexicographical complexity
$(b,p,a,c) =(b(S),p(S), a(S),c(S))$, where the entries are defined as follows:

$b(S)=|\bdry S\cap \C^1|$, the number of intersections of $\bdry S$ with the
1-skeleton of the induced triangulation on $\bdry M$.  This is the combinatorial
{\it length of $\bdry S$}.

$p(S)=|S\cap P^1|$, where  $P^1$ is the union of prism edges (1-prisms).

$a(S)=|S\cap (\C^1-(P^1\cup\bdry M))|$, the number of intersections with interior
non-prism edges.  This is a {\it combinatorial area}.

$c(S)$ is the sum over 2-prisms $\sigma$, meeting $\bdry M$ in a 1-simplex $\epsilon$, of the
number of distinct pairs
$(\alpha,\beta)$ of arcs of $S\cap \sigma$ such that
$\alpha$ and
$\beta$ are trivial arcs with both ends in $\epsilon$, and  $\beta$ is in the half
disk cut from $\sigma$ by
$\alpha$.   This is the {\it concentricity} of
$S$.  We say that $\beta$ is {\it concentric in} $\alpha$ but the relation is not
symmetric.

We call this complexity the {\it tunnel normal complexity} or just {\it tunnel
complexity}.

If $S$ has minimal complexity among all prenormal surfaces in its isotopy class, we
say $S$ is in {\it tunnel prenormal form
 of minimal complexity}. 
\end{defn}

Trivial arcs of intersection in a 2-prism, of the type we describe above, can be
thought of as ``tunnels" if we replace them by a tubular neighborhood in $S$.  The entry
$c(S)$ can be thought of as counting the number of pairs of ``concentric" tunnels.  As is
typical in normal surface theories, we now wish to enumerate the possible combinatorial
types of the disk components of intersection of $S$ with simplices
$\Sigma$. We will always assume that $S$ is tunnel prenormal of minimal complexity.

Before we begin analyzing possible intersections of a tunnel normal surface $S$ with
various types of 3-cells, let us establish some general principles concerning
patterns of intersection of a disk of $S\cap \Sigma$ with faces of $\Sigma$. 
\begin{lemma}\label{PrinciplesLemma}  Let $S$ be a tunnel prenormal surface of
minimal complexity.  Let
$\Sigma$ be a 3-cell.  

\hip
\noindent (i) If $\Sigma$ is a 3-prism and $\Sigma\cap \bdry M$ is a 2-simplex
$\sigma$, then
$\bdry S\cap \sigma$ consists of arcs, each joining different sides of
$\sigma$, so that
$\bdry S$ is a normal curve system relative to the triangulation of $\bdry M$.

\noindent (ii) If $e$ is any 1-prism, then $S\cap e=\emptyset$.

\noindent (iii) If $\Sigma$ is any 3-cell and $E$ is a disk of $S\cap  \Sigma$, then
$\bdry E$ intersects any interior edge of $\Sigma$ at most once.

\noindent (iv) If $\Sigma$ is a 3-prism, $R=\epsilon\times I$ is any 2-prism on the
boundary of $\Sigma$ and $E$ is a disk of
$S\cap
\Sigma$, then $\bdry E$ does not contain concentric trivial arcs in $R$.
\end{lemma}

{\bf Remark:} The following statement follows immediately from (ii)-(iv).   If
$\Sigma$ is a 3-prism, if $R=\epsilon\times I$ is any 2-prism on the boundary of
$\Sigma$ intersecting
$\bdry M$ in the edge
$\epsilon=\epsilon\times
\{0\}$, and if $E$ is a disk of
$S\cap
\Sigma$, then $\bdry E$ intersects $R$ in:

\noindent (1) a possibly empty set of trivial arcs (each with both ends in $\epsilon$)
containing no pair of concentric arcs, and

\noindent (2) at most one arc of the form $x\times I\subset \epsilon\times I=R$ with
$x=(x,0)\subset \epsilon\times I$ (up to isotopy).

\begin{proof}   (i) If $\sigma$ is a 2-simplex in $\bdry M$, and $\bdry S\cap
\sigma$ contains an arc with both ends in a side of $\sigma$, clearly  one can
reduce $b(S)$.  This amounts to showing that $\bdry S$ is normal in the
triangulation of $\bdry M$.
\hop

For any disk $E$ of $S\cap \Sigma$ as in the statements (iii),(iv), the disk can be
isotoped to either of the two disks $E_1$ and $E_2$ in $\bdry \Sigma$ bounded by
$\bdry E$. We can think of $E$ as being very close to $E_1$ or $E_2$, say $E_1$, but
still properly embedded, and we can choose a product structure between $E$ and
$E_1$. 
\hop

(ii) Suppose $e$ is a 1-prism, i.e. an edge of a prism $\Sigma=\sigma \times I$
contained in $\bdry \sigma\times I$.  If $S\cap e\ne \emptyset$ then an isotopy of
the intersection point along the edge away from $\bdry M$ eliminates the
intersection, possibly at the expense of increasing other entries.  The isotopy
should be done rel $\bdry S$, so $b(S)$ is not increased.   Notice that we achieve
$p(S)=0$.   

(iii) If $E$ is a disk in $\Sigma$ intersecting an interior edge more
than once, and we choose $E_1$ so that
$E_1$ contains a segment
$\alpha_1$ of the interior edge, with $\bdry \alpha_1\subset \bdry E$, we can isotope the
corresponding arc
$\alpha$ in
$E$  (corresponding via the product structure) to
$\alpha_1$ and beyond.  This eliminates two intersections of $S$ with an interior
edge, reducing $a(S)$.  If there are other dics of
$S\cap \Sigma$ in the product between $E$ and $E_1$, then we may be able to reduce
$a(S)$ further by eliminating other intervening arcs of intersection.  The isotopies
that we have performed may have the effect of replacing two essential arcs in
another 2-simplex (not on $\bdry \Sigma$) by an inessential arc.  The isotopies we
perform here do not affect intersections with 1-prism interior edges, so $p(S)=0$
remains true, and $b(S)$ also remains unchanged.   We have proved (iii).  For an
alternative proof which avoids the issue of additional disk components of
$S\cap \Sigma$ between
$E$ and $E_1$, assume that the disk $E$ is chosen so $E_1$ is innermost among disks
$E$ with the property that
$\bdry E$ intersects an interior edge in more than one point.

(iv) Suppose a disk $E$ of $S\cap \Sigma$ has the property that $\bdry E$ intersects
a 2-prism $R$ in $\bdry
\Sigma$ in concentric arcs.  We choose a disk $E_1$ in $\bdry \Sigma$
bounded by $\bdry E$ such that
$V=E_1\cap R$  is contained in one of the half-disks bounded by one of the arcs
$\gamma$ of $\bdry E\cap R$, but is not the entire half-disk, see Figure
\ref{TunnelConcentric}.

\begin{figure}[ht]
\centering
\scalebox{1.0}{\includegraphics{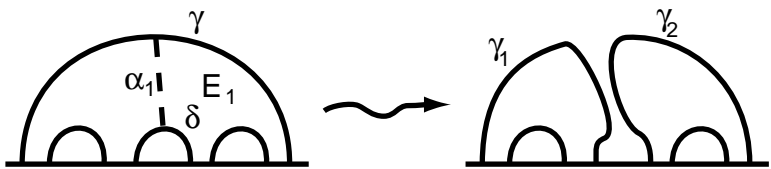}}
\caption{\small Surgering concentric arcs in a 2-prism.}
\label{TunnelConcentric}
\end{figure}

Then $\gamma$ represents a ``tunnel" which is not innermost.  We choose an arc
$\alpha_1$ joining $\gamma$ to any of the other arcs in $\bdry E_1\cap R$, say
$\delta$ concentric in
$\gamma$.  If $\alpha$ is the corresponding arc in $E$, and we isotope $\alpha$ to
$\alpha_1$, extending the isotopy to
$S$, the effect is to surger the pattern of intersection in $R$ as shown in the
figure replacing $\gamma$ and
$\delta$ by arcs $\gamma_1$ and
$\gamma_2$.  The effect on the concentricity entry of the complexity is to remove
the concentric pair
$(\gamma,\delta)$.   Also, if $\epsilon$ is concentric in $\delta$, then the surgery
removes the concentric pairs $(\gamma,\epsilon)$ and $(\delta,\epsilon)$.  For any
other concentric pair
$(\gamma,\epsilon)$, the surgery removes this pair but introduces either
$(\gamma_1,\epsilon)$ or
$(\gamma_2,\epsilon)$, resulting in no further change to $c(S)$.  We conclude that
if some $\bdry E$ contains concentric trivial arcs, then $c(S)$ can be reduced.  Clearly
this move does not change any of the previous entries of the complexity.
\end{proof}

In the above proof, we are working with a kind of combinatorial minimal surface.  It is
worth describing the analogue for Riemannian minimal surfaces at an heuristic level.  We
start by isotoping
$S$ so
$\bdry S$ becomes geodesic, then we fix $\bdry S$. Next we minimize the area of $S$ rel
$\bdry S$ by isotopy.  It is easy to understand how the entries $p(S)$ and $a(S)$
can be used to approximate a suitable Riemannian metric:  the metric must be chosen to
assign a large area to $\bdry M$ compared to areas of surfaces pushed to the interior of
$M$.  ``Tunnels" have relatively small area, but an isotopy which replaces two concentric
tunnels by two side-by-side tunnels, as in the proof of (iv), reduces Riemannian area.

Based on Lemma \ref{PrinciplesLemma} we can now enumerate the possible disk types
for a surface in tunnel prenormal form of minimal complexity.  In all of our
enumerations of disk types, it should be understood that symmetries of a 3-simplex
or 3-prism map disk types to other disk types, so we will only enumerate disk types
up to symmetries.  The disk types in 3-simplices will be the usual disk types in the
classical normal surface theory for triangulated 3-manifolds, see Figure
\ref{TunnelTypesSimplex}.  After including symmetric images of the disk types shown,
we obtain the usual 7 disk types. 
\begin{figure}[htp]
\centering
\scalebox{1.0}{\includegraphics{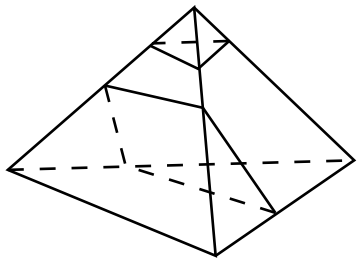}}
\caption{\small Disc types for 3-simplices of the cellulation.}
\label{TunnelTypesSimplex}
\end{figure}

\begin{figure}[htp]
\centering
\scalebox{1.0}{\includegraphics{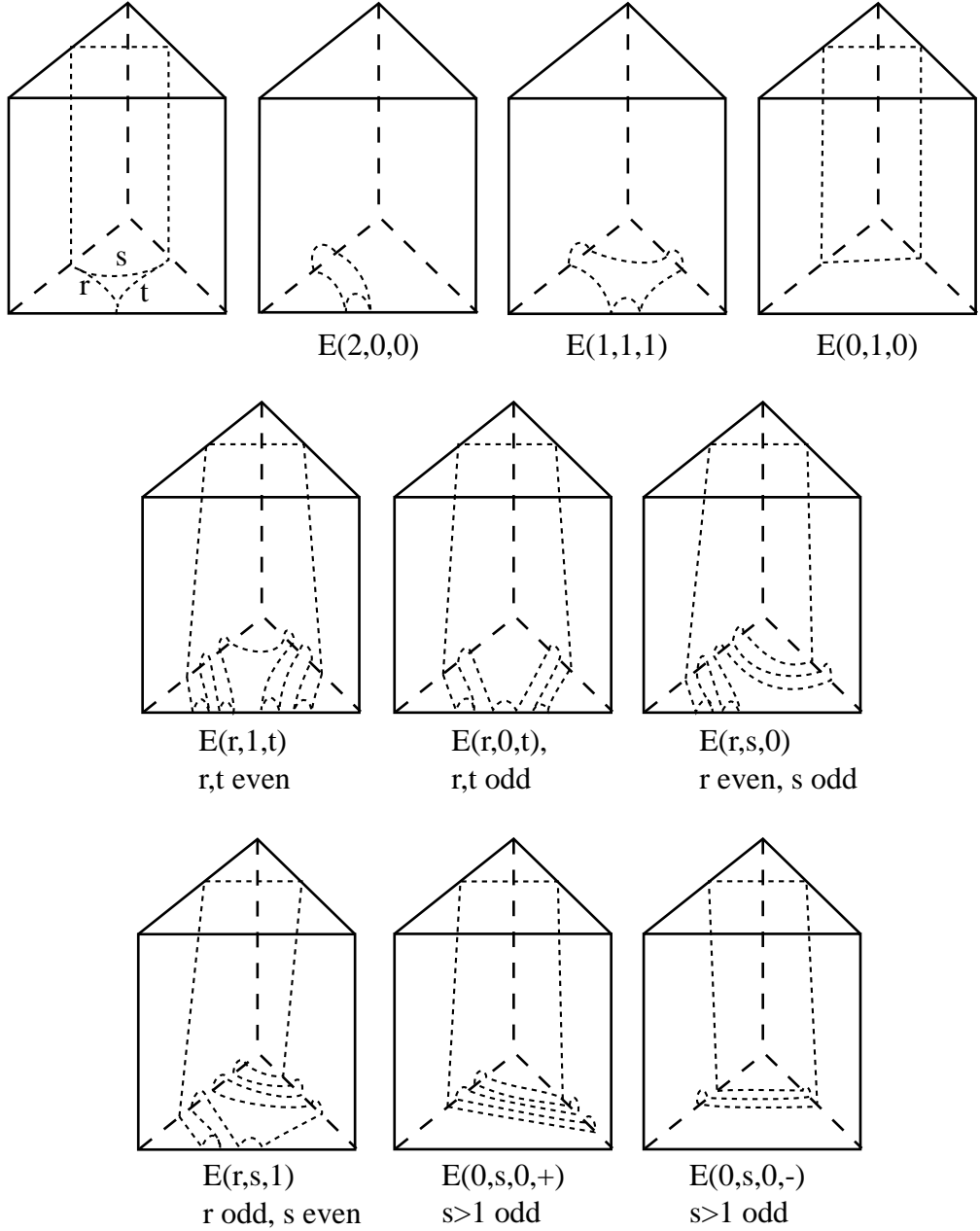}}
\caption{\small Disc types for 3-prisms of the cellulation.}
\label{TunnelTypesPrism}
\end{figure}

It is in the 3-prisms $\Sigma=\sigma\times I$ that we find infinitely many disk
types.  Recall that $\Sigma\cap
\bdry M=\sigma\times 0$.  Some of the disk types in a 3-prism $\Sigma$ are disjoint
from $\Sigma\times 1$.  These are shown in Figure \ref{TunnelTypesPrism}, labeled
$E(2,0,0)$ and
$E(1,1,1)$.  The remaining disk types intersect $\sigma\times 1$.  Lemma
\ref{PrinciplesLemma} shows that such disk type $E$ can intersect
$\sigma\times 1$ in only one essential arc joining different sides of $\sigma\times
1$.  Also, $\bdry E$ can intersect each 2-prism in at most one essential arc.   The
disk types are almost classified by the numbers of each of the essential arcs of
$\bdry E\cap (\sigma\times 0)$.  Fixing a particular view of the prism as shown, we
denote the weights induced on the essential arcs in $\sigma \times 0$ by $r,s,t$
respectively, as shown in Figure
\ref{TunnelTypesPrism}.  The disk types which arise and intersect $\sigma\times 1$
can all be regarded as being obtained by modification from the rectangular disk type
$E(0,1,0)$ shown in the figure.  We label all the disk types as $E(r,s,t)$ according
to the weights induced on the triangular train track in
$\sigma\times 0$.  There are two disk types inducing weights $(0,s,0)$, so we denote
them as $E(0,s,0,+)$ and
$E(0,s,0,-)$ as shown in Figure
\ref{TunnelTypesPrism}.  We further note that the rectangle $E(0,1,0)$ {\it could}
be regarded as a special case of $E(0,s,0,+)$ or $E(0,s,0,-)$ with $s=1$.  Or it
could be regarded as $E(r,1,t)$, with $r=t=0$, or as $E(r,s,0)$ with $r=0$ and
$s=1$.    Clearly not all triples $(r,s,t)$, with $r$, $s$ and $t$ all positive
integers, are induced by a disk type, but all such triples are induced by disjoint
unions of disks belonging to different disk types, usually in many different ways. 
\begin{defn} A {\it tunnel disk type} is a disk type in a 3-simplex $\Sigma$
isomorphic to one shown in Figure
\ref{TunnelTypesSimplex} or a disk type in a 3-prism isomorphic to one shown in
Figure \ref{TunnelTypesPrism}.  A surface $S\embed M$ is in {\it tunnel normal form}
if it intersects 3-simplices and 3-prisms only in disks belonging to tunnel disk
types.
\end{defn}

\begin{proposition}\label{SubToNormalProp}  A surface $S$ in tunnel prenormal form
of minimal complexity is in tunnel normal form.
\end{proposition}

\begin{proof} We give a sketch.  Given a tunnel prenormal surface $S$ with respect
to a cellulation, suppose
$E$ is any disk of $S\cap \Sigma$  in a 3-simplex or 3-prism $\Sigma$.  For a
3-simplex, the condition that
$\bdry E$ can intersect each edge at most once easily gives the types of disks $E$
shown in Figure
\ref{TunnelTypesSimplex}, the same disk types that occur in the classical normal
surface theory.   To analyze the possible disk types in a 3-prism
$\Sigma=\sigma\times I$, one considers the possible closed curves $\bdry E$ in
$\bdry \Sigma$ such that the disk $E$ bounded by the closed curve satisfies the
conditions of Lemma \ref{PrinciplesLemma}. 

\begin{figure}[htp]
\centering
\scalebox{1.0}{\includegraphics{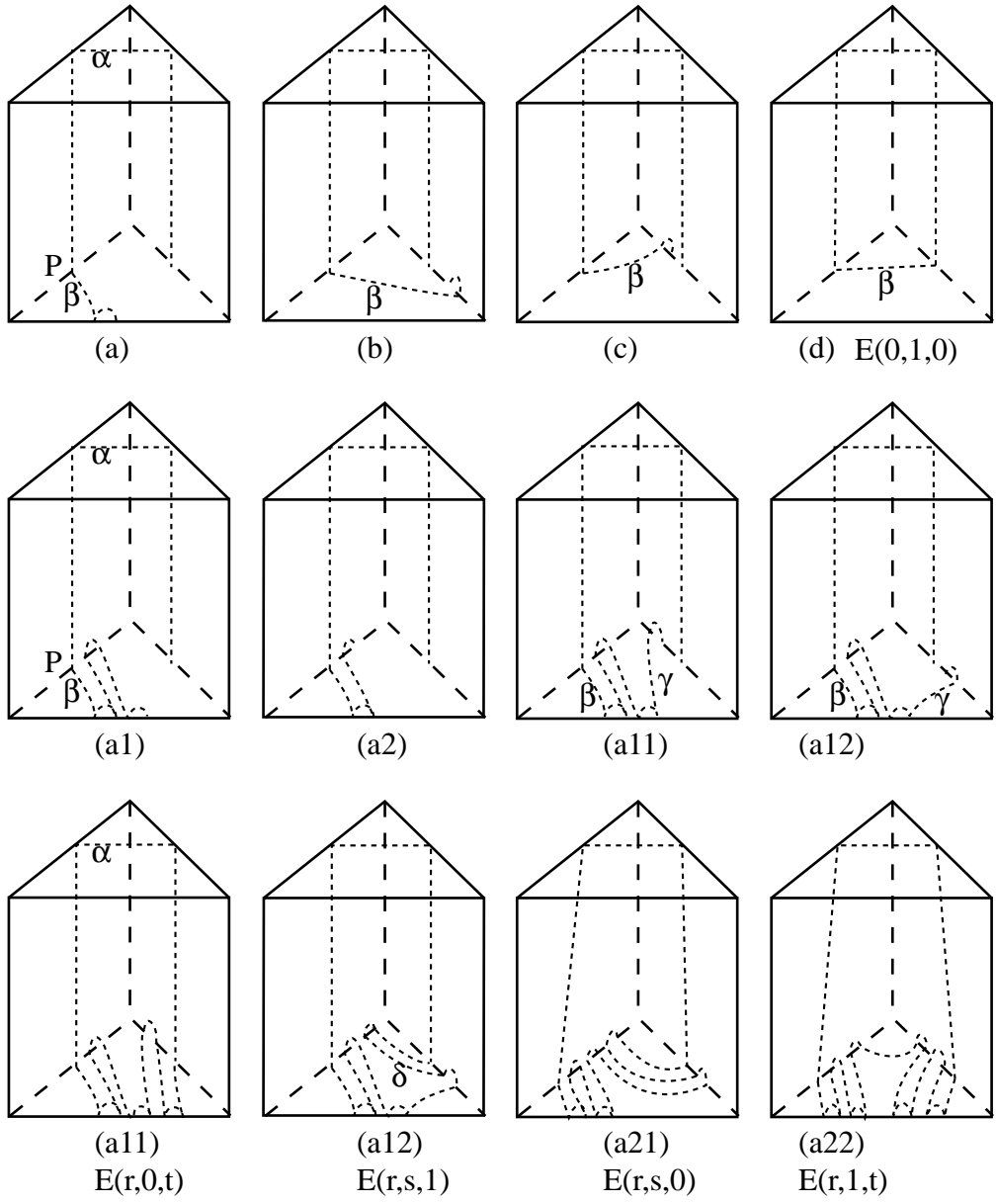}}
\caption{\small Enumerating tunnel disk types in a 3-prism.}
\label{TunnelTypesInduction}
\end{figure}

In the easiest case, we assume that $\bdry E$ is disjoint from $\sigma\times 1$. 
This means that $\bdry E$ can be subdivided into arcs which are essential in
$\sigma\times 0$ and inessential arcs, each lying in a 2-prism
$\epsilon\times I$, where $\epsilon$ is an edge in $\sigma$, and each having both
ends in $\epsilon \times 0$.  If one avoids concentric trivial arcs in $\bdry E$, as
required by Lemma \ref{PrinciplesLemma}, the only possibilities are shown in Figure
\ref{TunnelTypesPrism}, labelled $E(2,0,0)$ and $E(1,1,1)$.  In the remaining cases,
$\bdry E$ intersects $\sigma\times 1$ in just one arc, $\alpha$ say, and the ends of
this arc are connected to two essential vertical arcs, see Figure
\ref{TunnelTypesInduction}.  In our standard view, we let $P$ denote the endpoint of
the left vertical arc, as shown in the figure.  Now we consider possible extensions
of our path from $P$ by an essential arc $\beta$ in $\sigma\times 0$.  We show the
four possibilities for $\beta$ in (a), (b), (c), (d) of the figure.  In case (d), we
close the path to obtain the rectangular disk type
$E(0,1,0)$.  By Lemma
\ref{PrinciplesLemma}, in all other cases the next extension must be by a trivial
arc in one of the 2-prism faces
$\epsilon
\times I$, and one easily sees that only one of the choices is possible, as shown,
otherwise the path cannot be closed.

We must now pursue each of the possibilities indicated in (a), (b), and (c) of the
figure.  To demonstrate the method for exhausting possibilities, we will just pursue
(a).  To begin with, we can continue with a further arc isotopic to $\beta$, then
another trivial arc, then another $\beta$, etc., repeating the zig-zag pattern as
often as we wish as shown in the figure.    But the zig-zag path might end on the
edge containing $P$, see (a2), or it might end on the front bottom edge, see (a1). 
We now pursue the possibility shown in (a1).  There are two possibilities for the
next essential arc $\gamma$ not parallel to $\beta$ in $\sigma\times 0$ and for the following
inessential arc, as shown in (a11) and (a12).  Pursuing (a11), there is only a choice how
often one can zig-zag along
$\gamma$ before closing the path.  In fact, it is possible that we close the path
after traversing $\gamma$ only once.  Next, pursuing the possibility (a12), there
can be no zig-zaging along
$\gamma$; there is only one choice for the next essential arc $\delta$ as shown, and
there can be arbitrarily many zig-zags along $\delta$ before closing the path.

Now we can return to the case (a2), which leads to case (a21) and (a22) as shown.  

Clearly, we have not finished the enumeration of disk types, but it should now be
easy for the reader to fill in the remaining details of the argument.  Note that as
one follows the tree of choices for closing the paths, one arrives at some disk
types which have already appeared (up to a symmetry of the prism).  That explains
why we have already found a majority of the disk types.  
\end{proof}

\begin{proof}[Proof of Theorem \ref{NormalThm}]  Let $S$ be a semi-essential
surface.  We put it in prenormal form of minimal complexity.  Then by Proposition
\ref{SubToNormalProp}, it is in tunnel normal form.
\end{proof}

\section{Carriers}\label{Carriers}

From a classical least area normal surfaces (in the sense of Haken)  one obtains a
normal incompressible branched surface which carries it, see
\cite{WFUO:IncompressibleViaBranched}.   We now construct the analogous object for
tunnel normal surfaces. We call it a ``normal carrier,"  and describe it by
describing the possible intersections with 3-simplices and 3-prisms of our
cellulation.    Suppose $S$ is a tunnel normal surface.  If $\Sigma$ is a 3-simplex,
then $S\cap \Sigma$ is a disjoint union of disks belonging to finitely many disk
types.  We construct a {\it local carrier} for $S\cap \Sigma$, obtained by isotoping
$S\cap \Sigma$ and identifying disks of the same disk type, while also identifying
arcs in each 2-dimensional face of $\Sigma$ if they belong to the same arc type. 
The result is one of several branched surfaces like the one shown in Figure
\ref{TunnelLocalSimplex}.   The ``branching" occurs where disks of different disk
types sharing a common arc type on the boundary are identified, and we suppose that
where the disks are identified along a common arc, the tangent planes agree.  In
fact, for technical reasons, it is better to identify disk types in a regular
neighborhoods of identified arcs.  In the example of a local carrier shown in the
figure, three different disk types are represented.  There are examples with up to
five different disk types represented.  These examples of local carriers in a
simplex can also be called {\it local branched surfaces}, a special case of a local
carrier. 
\begin{figure}[htp]
\centering
\scalebox{1.0}{\includegraphics{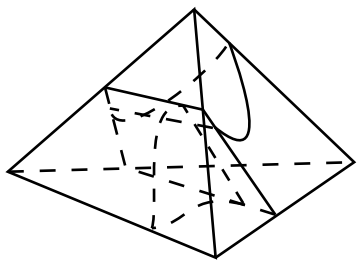}}
\caption{\small Example of a local carrier in a 3-simplex.}
\label{TunnelLocalSimplex}
\end{figure}

A local carrier in a 3-prism $\Sigma=\sigma\times I$ is a mix of branched surface
and train track, with parts of the local carrier being 1-dimensional, and other
parts 2-dimensional.  Starting with the intersection of
$S\cap\Sigma$, we begin by constructing a {\it local branched surface} in $\Sigma$,
just as we did in a tetrahedron.  Namely, we identify arcs of the same essential arc
type in any face of $\Sigma$ extending the identification to a regular
neighborhood.  Then we identify disks of the same disk type.  At this point we have
a ``branched surface" with possibly countably infinitely many sectors.  Also,
observe that the branched surface may not be embeddable.  Identifications of arcs in
$\sigma\times 0$ change ``tunnels" to tubes which appear as closed curves in
vertical faces $\epsilon \times I$ intersecting $\epsilon\times 0$ at a single
point, as shown in Figure
\ref{TunnelCollapse}.  Thus a trivial arc of intersection intersecting $\epsilon
\times 0$ in two points becomes a closed curve of intersection intersecting
$\epsilon \times 0$ in a single point.  The example shown in the figure shows the
identifications that result when
$S\cap
\Sigma$ consists of a single disk of type
$E(4,3,0)$.  The final step to obtain the carrier is to collapse the trivial closed
curves in each face
$\epsilon
\times I$ to a single point.  When this produces topological spheres in the
resulting complex, the spheres are collapsed to essential arcs in $\sigma \times
0$.  Note that in our example we obtain a local carrier consisting of a rectangle
and an edge.  We choose an embedding of the carrier such that different manifold
parts of the carrier meet tangentially where they are identified.  As usual, it is
better to extend the identification from $\bdry \Sigma$ inwards in a collar, though
for simplicity we do not always illustrate this extended identification in the
figures.  Putting together the local carriers in different simplices and prisms by
identifying arcs of the same types in faces, we obtain a {\it normal carrier} which
{\it carries}
$S$.  The carrier is chosen with a smooth structure such that the intersection with
$\bdry M$ is a train track, and the intersection with the interior of $M$ is a
branched surface. 
\begin{figure}[htp]
\centering
\scalebox{1.0}{\includegraphics{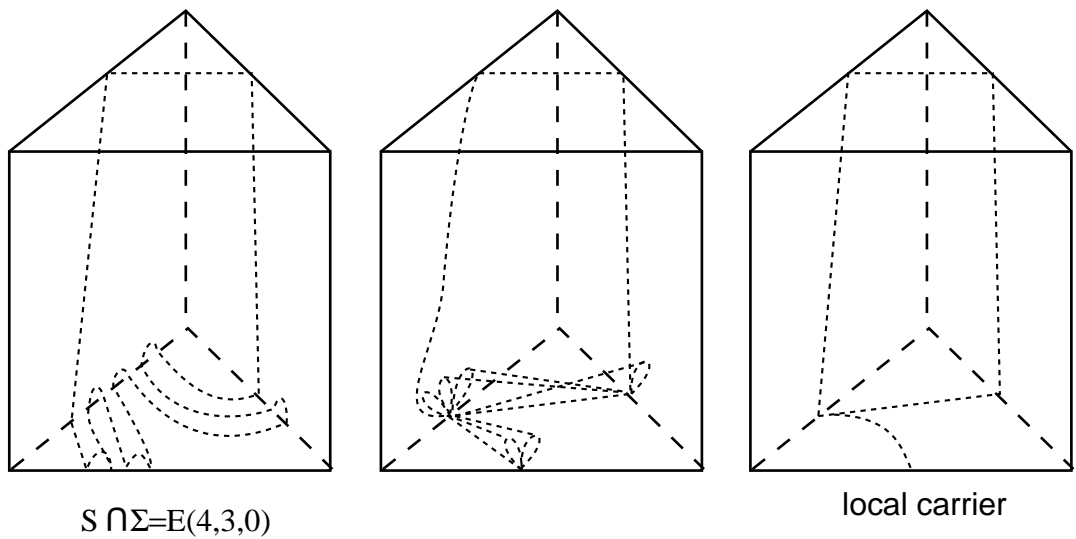}}
\caption{\small Construction of a local carrier in a 3-prism.}
\label{TunnelCollapse}
\end{figure}

In general, starting with any $S\cap \Sigma$ the local carrier obtained in a prism
is, up to symmetries, a sub-carrier of
$C_3=C_1\cup C_2$, where $C_1$ and $C_2$ are shown in Figure
\ref{TunnelLocalPrism}, and where arcs of $\bdry C_1$ and $\bdry C_2$ of the same
type in each face are identified to obtain $C_3$.   The carrier
$C_1$ consists of $\tau\times I$, where $\tau$ is a triangular train track in
$\sigma$, together with a triangular sector bounded by
$\tau\times 0$.  This triangular sector has a half-funnel shaped cusp at each of the
three vertices.   (The quality of the figure leaves something to be desired.)  Note
that the triangular sector at the bottom of $C_1$ is tangent to vertical sides of
$C_1$, which are tangent to each other where they meet.  The carrier
$C_2$ consists of a single disk with a half-funnel shaped cusp at one point on its
boundary.   Any local carrier is a sub-carrier of 
$C_3=C_1\cup C_2$.  We show some examples of local carriers in Figure
\ref{TunnelLocalPrism}.  The reader can verify that every tunnel disk type, and
every disjoint union of tunnel normal disks in a 3-prism or 3-simplex can be
isotoped respecting skeleta into a regular neighborhood of a local carrier.   
\begin{figure}[htp]
\centering
\scalebox{1.0}{\includegraphics{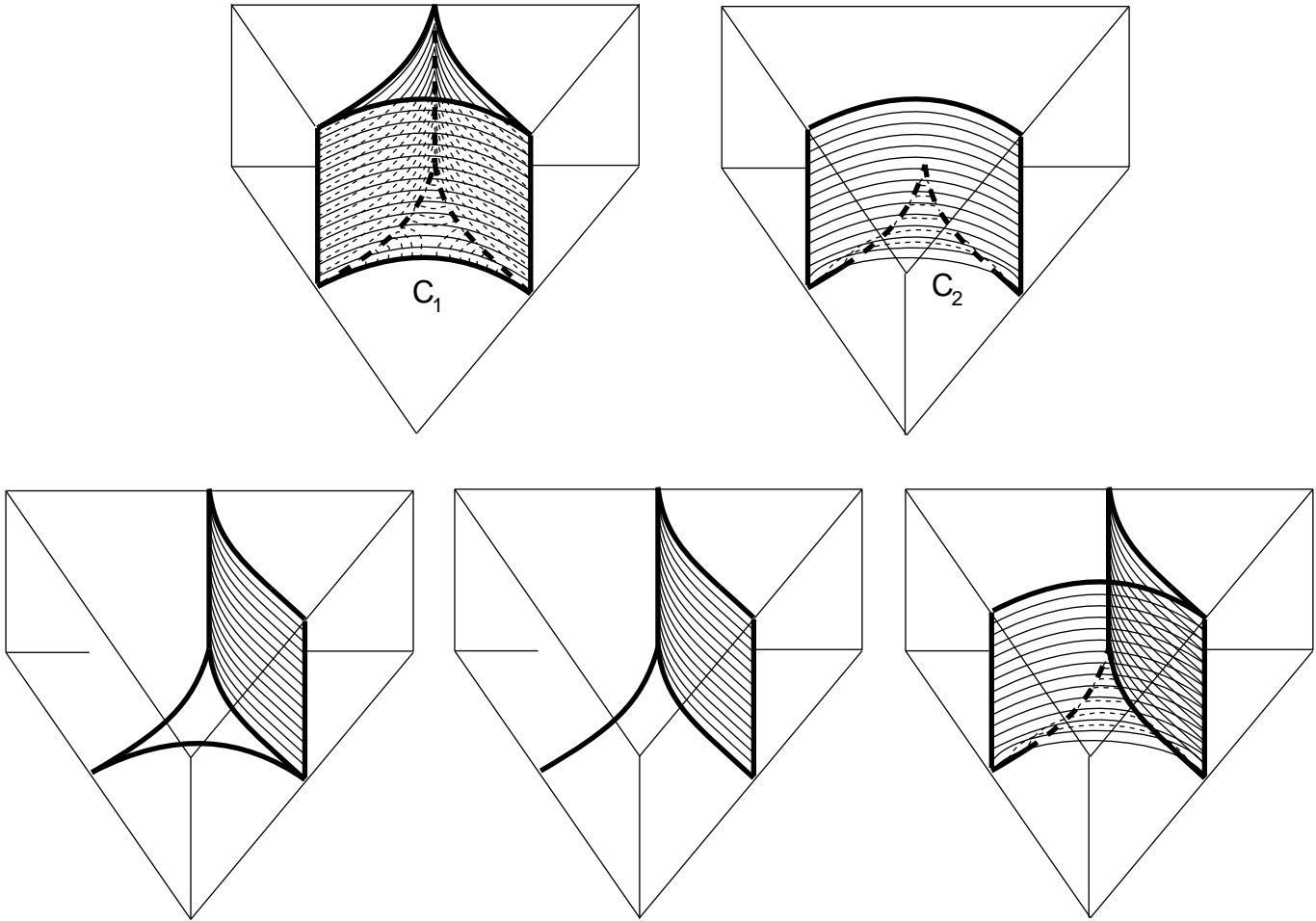}}
\caption{\small Examples of local carriers in a 3-prism.}
\label{TunnelLocalPrism}
\end{figure}

\begin{defn} A {\it normal carrier} in $M$ is a union $C$ of local carriers, one in
each 3-simplex or 3-prism, such that the closure of $C-\bdry M$ is a properly
embedded branched surface, and $\bdry C$ is a train track. 

A tunnel normal surface $S$ is {\it carried} by $C$ if it can be isotoped respecting
skeleta (i.e. isotoped without changing the types of disk intersections with
3-cells) into a regular neighborhood $N(C)$ of $C$.  The surface $S$ is {\it fully
carried} by $C$ if it is not carried by a proper subcarrier.
\end{defn}

Let us say that {\it canonical open regular neighborhood} $U(\bdry M)$ with respect
to our given cellulation is the union over 2-simplices $\sigma$ of the triangulation
of $\bdry M$ of prisms $\sigma\times [0,1)$.  Then it should be clear that $\bdry S$
is carried by $\bdry C=C\cap \bdry M$ in the usual sense, and $S-U(\bdry M)$ is
carried by the branched surface $C-U(\bdry M)$.

\begin{proof}[Proof of Theorem \ref{FiniteThm}] (i) Putting a semi-essential surface
$S$ in tunnel normal form of minimal complexity, we consider the intersection with a
single 3-simplex or 3-prism.  It intersects the 3-cell in finitely many disks
belonging to one of the tunnel normal disk types.  This union is fully carried by
one of the finitely many local carriers.  Glueing the local carriers for all the
3-cells, we obtain a carrier
$C$ which fully carries $S$.  There are just finitely many possibilities for the
local carriers, so there are just finitely many possibilities for $C$.  

(ii) We have already observed that if we put $S$ in tunnel normal form of minimal
complexity, then $\bdry S$ is normal and has minimal combinatorial length.  This
means that the carrier $C$ fully carrying $S$ has boundary $\bdry C$ an essential
train track (without complementary monogons or 0-gons).  Hence any curve system
carried by $\bdry C$ is essential.  Furthermore, for any essential curve system $L$
in $\bdry M$, there is a semi-essential surface $S$, consisting of annuli, isotopic to $N(L)$,
a regular neighborhood of $L$ in $\bdry M$.  Thus $L$ is carried by $\bdry C$, though not
necessarily fully carried.  

(iii) To check that a disk is essential, it is enough to
check that its boundary is essential.  Then (ii) shows that any disk carried by a
$C_i$ is essential.
\end{proof}

\section{Semi-essential non-disk surfaces \label{Non-disk}} 

In this section we consider arbitrary semi-essential surfaces, but the main purpose
is to give sufficient conditions for a non-disk surface to be semi-essential.  As a
preliminary, we will need to describe characteristic compression bodies in a
3-manifold. These were introduced by Francis Bonahon, see
\cite{FB:CompressionBody}.  Our version is a little more general than Bonahon's.  

Let $M$ be an irreducible, orientable 3-manifold with $W\embed \bdry M$ be a compact
essential subsurface of
$\bdry M$.  If
 ${\cal D}=\{D_1, D_2,
\ldots, D_q\}$ is a maximal collection of non-isotopic disjoint compressing disks of
$W$ in
$M$, a {\it characteristic compression body $Q$ associated to $W$} is defined to be
a regular neighborhood
$N=N(W\cup\cal D)$, with boundary spheres capped off by the balls they bound in
$M$.  Here we abuse notation by using
$\cal D$ also to denote
$\cup_iD_i$.  The characteristic compression body is a pair $(Q,U)$, where $U=\bdry
Q-\inter W=\bdry_iQ$, and
$W=\bdry_eQ$. 

The following proposition (a slight generalization of a result of F. Bonahon, see
\cite{FB:CompressionBody}) gives the essential properties of characteristic
compression bodies. 
\begin{proposition} \label{CharCompressionPropertiesProp} Let $M$ be an irreducible,
orientable 3-manifold, and
$W$ an essential surface in
$\bdry M$.

(i) If $Q$ is a characteristic compression body associated to
$W$, then
$U=\bdry_iQ$ is incompressible in $M$, but possibly with disk components.

(ii) The characteristic compression body $Q$ associated to $W$ is uniquely
determined up to isotopy.
\end{proposition}

For a proof, see \cite{UO:Autos}.  The surface $U$ in the above may include disk
components.  If we remove these from $U$ we obtain a surface $V$.

Corresponding to a choice of cell in the curve complex of $\bdry M$, or a primitive
curve system $X$ as described in the introduction, we obtain another primitive curve
system $\hat X$ by eliminating from $X$ all curves bounding disks.  We let
$A=N(\hat X)$, a regular neighborhood in $\bdry M$, and we let 
$W=\bdry M-\intr(A)$, which is an essential subsurface of
$\bdry M$, without disk components.  Then we let
$(Q_X,V)$ be the characteristic compression body associated to $W$.  By
construction, the interior boundary $V$ has no disk components.  Then we obtain the
complementary 3-manifold $T_X$, the closure of $M-Q_X$.  We regard
$T_X$ as a Haken pair 
$(T_X,A)$, but observe that it could equally well be regarded as a Haken pair
$(T_X,V)$, since both $A$ and $V$ are incompressible in $T_X$.  Let us call the pair
$(T_X,A)$ the {\it $X$-core of $M$}, and let $(Q_X,V)$ be the {\it
$X$-characteristic compression body}. 

\begin{proof}[Proof of Theorem \ref{SeparationThm}]  The fact that $K$ can be
isotoped into the characteristic compression body
$Q_X$ follows from the incompressibility of $V$ and the irreducibility of $M$, $Q_X$
(and $T_X$); one eliminates curves of intersection of
$K$ with
$V$.

Now we shall prove that $F$ can be isotoped into $T_X$.   We choose a union $H$ of
disks including $K$ such that cutting $Q_X$ on $H$ yields a product
$V\times I$ and possibly some balls.  Again using irreducibility, we eliminate
intersections of $F$ with $H$, then $F\cap Q_X$ lies in a product
$V\times [0,1]$ with $V=V\times 0$.  Further, $F$ is disjoint from $V\times 1$, so
we can isotope $F$ out of the product using the product structure.  Then $F$ is
semi-essential in $T_X$.  
\end{proof}

We can then analyze the semi-essential surfaces in $(T_X,A)$ using ordinary
incompressible branched surfaces:

\begin{proof}[Proof of Theorem \ref{Non-DiscThm}] (i) Suppose $S$ is semi-essential
in $(T_X,A)$ and suppose it is
$\bdry$-compressible.  Let
$H$ be a
$\bdry$-compressing disk with arc $\alpha=H\cap A$, $\beta=H\cap S$.  There are two
cases.  If $\alpha$ connects a closed curve of
$\bdry S$ to itself, then by isotoping $\alpha$ in $\bdry M$ to an arc in $\bdry S$,
and extending the isotopy to $H$, $H$ becomes a potential compressing disk.  Since
$S$ is incompressible, $\bdry H$ bounds a disk in $S$, which implies $H$ was not a
$\bdry$-compressing disk.  The remaining possibility is that $\alpha$ connects two
different curves $\gamma_1$ and
$\gamma_2$ of $\bdry S$.  We think of $\gamma_1$ and $\gamma_2$ as paths, each
starting and ending a the same point of
$\alpha\cap \bdry S$.  With all arcs appropriately oriented,
$\beta\gamma_2\beta\inverse\gamma_1\inverse$ is a closed curve in $S$ which bounds a
potential compressing disk constructed from two parallel copies of $H$ and a
rectangular region obtained by cutting on $\alpha$ the annulus between $\gamma_1$
and $\gamma_2$.  Since $S$ is incompressible, this closed curve bounds a disk in
$S$, which implies that the component of $S$ containing
$\beta$ is a $\bdry$-parallel annulus.  (If the reader has not seen this classical
argument, we suggest drawing a picture.) 

(ii) By (i), if we diskard $\bdry$-parallel annulus components of $S$, we can now
suppose $S$ is incompressible and
$\bdry$-incompressible.  This means that we can apply classical normal surface
theory exactly as in
\cite{WFUO:IncompressibleViaBranched} to obtain the result in (ii).  The only
difference is that we are dealing with incompressible surfaces for a Haken pair
$(T_X,A)$, but the construction of the branched surfaces is exactly the same, via
normal surface theory.  One uses a triangulation which induces a triangulation of
$A$.  

(iii) The results of \cite{WFUO:IncompressibleViaBranched} also imply that any
surface fully carried by one of the
$B_{(X,j)}$'s is incompressible and $\bdry$-incompressible.  In order to prove that
{\it any\ }  surface carried by one of these branched surfaces is incompressible and
$\bdry$-incompressible, one must use the methods of
\cite{UO:MeasuredLaminations}. 
\end{proof}

\begin{proof}[Proof of Cor \ref{SeparationTwoCor}] (a)  Recall that for this theorem
we  assume $Y$ is primitive with the property that no curve of $Y$ bounds a disk in
$M$.  This means
$\hat Y$ is the same as
$Y$, and we define $Q_Y$ and $T_Y$ as before.  In Theorem \ref{SeparationThm}, with
$X=Y$,
$S=F\cup K$, but $K$, which is a union of disks, must then by empty, so $S$ is
incompressible in $T_Y$.    (b) Here $\S_Z$ consists of disks only, $\hat
Z=\emptyset$ so $S=K$ can be isotoped into $Q_Z$, which is the entire 
characteristic compression body.
\end{proof}

\bibliographystyle{amsplain}
\bibliography{ReferencesUO3}

\end{document}